\documentclass[a4paper,11pt]{amsart}
\usepackage{amsmath,amssymb,amsthm}

\newcommand{\IM}{\operatorname{\mathcal{M}}}

\newcommand{\LL}{\operatorname{\mathcal{L}}}

\newcommand{\PP}{\operatorname{\mathfrak{P}}}

\newcommand{\Si}{\dot{S}}

\newcommand{\CR}{\bar{\partial}}

\newcommand{\ev}{\operatorname{ev}}

\newcommand{\QH}{\operatorname{QH}}

\newcommand{\del}{\partial}

\newcommand{\IC}{\operatorname{\mathbb{C}}}

\newcommand{\IR}{\operatorname{\mathbb{R}}}
\newcommand{\IN}{\operatorname{\mathbb{N}}}

\newcommand{\Ih}{\operatorname{\mathbf{h}}}

\newcommand{\Q}{\operatorname{Q}}
\newcommand{\V}{\operatorname{V}}
\newcommand{\T}{\operatorname{\mathcal{T}}}
\renewcommand{\LL}{\operatorname{\mathcal{L}}}
\renewcommand{\P}{\operatorname{\mathcal{P}}}

\newtheorem{theorem}{Theorem}[section]
\newtheorem{proposition}[theorem]{Proposition}

\newtheorem{corollary}[theorem]{Corollary}

\title{Note on Floer theory and \\integrable hierarchies}
\author{Oliver Fabert}
\thanks{O. Fabert, VU Amsterdam, The Netherlands. Email: oliver.fabert@gmail.com}
\begin{document}

\maketitle

\begin{abstract}
In this short note we show how Dubrovin's integrable hierarchies, defined using the Gromov-Witten theory of a closed symplectic manifold, generalizes to Hamiltonian Floer theory. In particular, we show how the required generalization of the PSS isomorphism, relating Gromov-Witten theory and Hamiltonian Floer theory, can be constructed in the framework of Eliashberg-Givental-Hofer's symplectic field theory\end{abstract}

\tableofcontents
\markboth{O. Fabert}{Floer theory and } 

\section*{Summary}
The Floer theory of Hamiltonian symplectomorphisms is an important tool in symplectic geometry. Floer cohomology was invented by A. Floer to prove the Arnold conjecture about the number of symplectic fixed points and since then was improved to answer many other questions in symplectic geometry. Following M. Schwarz and P. Seidel (\cite{Sch}), there exists the so-called pair-of-pants product in Floer cohomology. Apart from the Arnold conjecture for degenerate Hamiltonians, it is used in proofs of the Conley conjecture and plays a crucial role in the definition of symplectic quasimorphisms. While the pair-of-pants product involves the Floer cohomology groups of different Hamiltonian symplectomorphisms, it in particular defines a graded commutative and associative product on the sum of the Floer cohomologies of all powers of a given Hamiltonian symplectomorphism $\phi$.\\

The applications of the pair-of-pants product build on its relation with the small quantum product of the underlying symplectic manifold. Following Piunikhin-Salamon-Schwarz (\cite{PSS}) there exists an ring isomorphism between Floer cohomology with its pair-of-pants product and the small quantum cohomology ring. On the other hand, the small quantum product only involves a very small part of rational Gromov-Witten theory, since it just counts holomorphic spheres with three marked points. In order to use the geometric information of all rational Gromov-Witten invariants, it is known, see \cite{MDSa}, that there also exists a big version of the quantum cup product, which recovers the full rational Gromov-Witten potential. Following B. Dubrovin, it is the main ingredient to equip the quantum cohomology with the structure of a Frobenius manifold. \\
    
In \cite{FF1}, \cite{FF2} we show how the big quantum product and the Frobenius manifold structure translate to the Floer theory of a Hamiltonian symplectomorphism $\phi$, extending the relation between the small quantum product and the pair-of-pants product. On the other hand, Dubrovin has shown, see \cite{DZ}, that one can assign an integrable hierarchy to each Frobenius manifold. In this note we show how Dubrovin's integrable hierarchy translates to Hamiltonian Floer theory and show how the required generalization of the PSS isomorphism can be defined using the algebraic formalism of Eliashberg-Givental-Hofer's symplectic field theory.

\section{Symplectic field theory of Hamiltonian mapping tori}

The goal of this note is to show how to define Dubrovin's integrable hierarchy defined on the loop space of the small quantum cohomology ring of a closed symplectic manifold generalizes to Hamiltonian Floer theory. The resulting new algebraic structures will be defined in an extension of Eliashberg-Givental-Hofer's symplectic field theory for mapping tori. \\

For this we first quickly review the geometric setup for the symplectic field theory of Hamiltonian mapping tori; for details we refer to \cite{FF1}. We start with the observation that (parametrized) one-periodic Hamiltonian orbits $x: S^1\to M$ are in one-to-one correspondence with unparametrized one-periodic orbits $\gamma$ of the canonical vector field $\del_t$ on the corresponding mapping torus $M_{\phi}=\IR\times M/\{(t,p)\sim (t+1,\phi(x))\}$ by setting $\gamma: S^1\to M_{\phi}$, $\gamma(t)=(t,x)$ where $x$ is viewed as the corresponding fixed point. Following (\cite{E},example 1.2), see also \cite{F1}, note that $M_{\phi}$ naturally carries a stable Hamiltonian structure in the sense of \cite{BEHWZ} given by $(\tilde{\omega}=\omega,\tilde{\lambda}=dt)$  with Reeb vector field $\tilde{R}=\del_t$. As described in \cite{F1}, the stable Hamiltonian manifold $M_{\phi}$ can be identified with $S^1\times M$ equipped with the $H$-dependent stable Hamiltonian structure $(\tilde{\omega}^H=\omega+dH_t\wedge dt,\tilde{\lambda}^H=dt)$ with Reeb vector field $\tilde{R}^H=\del_t+X^H_t$, where the underlying diffeomorphism between $M_{\phi}$ and $S^1\times M$ is given by the Hamiltonian flow, $S^1\times M\to M_{\phi}$, $(t,p)\mapsto (t,\phi^t_H(p))$. \\

While the closed orbits of period one are in bijection with the fixed points $x$ in $\P(\phi)$, note that for general $k\in\IN$ the fixed points in $\P(\phi^k)$ are in $k$-to-one-correspondence with closed orbits of period $k$ when the underlying orbit is simple. For this observe that for every fixed point $x\in\P(\phi^k)$ the points $\phi(x),\ldots,\phi^{k-1}(x)$ are also fixed points of $\phi^k$.  Note that the Conley-Zehnder indices (mod $2$) of $\phi^i(x)$ agree for all $i=0,\ldots,k-1$ by symmetry reasons, since the corresponding one-periodic orbits just differ by reparametrization and the spanning surface $u$ for $x$ naturally defines spanning surfaces for all $\phi^i(x)$. On the other hand, $x,\phi(x),\ldots,\phi^{k-1}(x)$ all represent the same unparametrized $k$-periodic Reeb orbit $\gamma$. Without further mentioning, we will only consider closed Reeb orbits where the underlying parametrized orbits in $M$ are contractible. Then the Conley-Zehnder index of $\gamma$ defined in \cite{EGH} agrees with the Conley-Zehnder index of $x$ (we can use the same spanning surfaces to define the index for $\gamma$).  More precisely, there is indeed a one-to-one correspondence between the $k$ fixed points and the $k$ special points that we have chosen on $\gamma$ above. \\

\section{Floer theory and commuting Hamiltonian systems in SFT}

It was shown by B. Dubrovin, see \cite{DZ}, that to every Frobenius manifold $\operatorname{Q}$ one can assign an infinite-dimensional integrable system on the loop space $\Lambda\operatorname{Q}$ of $\Q$. By definition, it consists of an infinite sequence of linearly independent commuting Hamiltonian functions which span the space of symmetries of the first Hamiltonian, see \cite{DZ} for the precise definition. For this recall that in Gromov-Witten theory the integrable system appears as flat coordinates for the deformed flat connection $\tilde{\nabla}$ on the cotangent bundle to the Frobenius manifold $\Q$ times $\IC^*$ given by the flat metric on $\Q$, the big quantum product $\star$ and the Euler vector field $E$. \\ 

While we show in \cite{FF2} that the big pair-of-pants product defines a (1,2)-tensor field $\star\in\T^{(1,2)}\Q_X$ on the differential graded manifold of contact homology, it is not clear how to define the analogue of Dubrovin's flat connection as other structures such as the Euler vector field $E\in\T^{(1,0)}\Q$ and the canonical flat structure do \emph{not} descend to well-defined structures on the differential graded manifold $\Q_X$ in general. As a consequence, the classical approach to integrable systems does not generalize immediately from Gromov-Witten theory to Floer theory. As a bypass we will use that, as outlined by Y. Eliashberg in his ICM plenary talk \cite{E}, the integrable system of the Gromov-Witten theory of $(M,\omega)$ naturally arises in the rational symplectic field theory (SFT) of the mapping torus $S^1\times M$ with vanishing Hamiltonian, see the paper \cite{Ro} of P. Rossi. Instead of trying to generalize the classical approach starting from the cohomology F-manifold, we will follow the latter approach as it leads in a much more natural way to the desired generalization of the integrable systems. \\

Indeed, using symplectic field theory one gets an infinite system of commuting Hamiltonians on the rational SFT of the mapping torus of every symplectomorphisms on a closed symplectic manifold, which agrees with the integrable system from Gromov-Witten theory in the case when the symplectomorphism is the identity. We emphasize that this observation was the starting point for our project of relating Floer theory, Frobenius manifolds and integrable systems, and even guided us to our definition of the big pair-of-pants product and its relation to cohomology F-manifolds. While for the definition of the cohomology F-manifold we used the full contact homology of mapping tori, the commuting Hamiltonian systems naturally live on their rational SFT homology, whose differential now counts holomorphic curves with an arbitrary number of positive (and negative) cylindrical ends and whose chain complex contains the full contact homology complex as a subcomplex. The reason is that the rational SFT homology naturally carries a Poisson bracket, which in turn leads to the natural appearance of commuting Hamiltonian systems in SFT, see \cite{E} and \cite{F3}. Apart from the fact that we claim that the cohomology F-manifold structure on contact homology indeed can be extended to rational SFT after introducing additional marked points on the curves (in order to model nodal breakings which now need to be included), the system of commuting Hamiltonian functions on rational SFT still restricts to a system of commuting vector fields on the differential graded manifold of full contact homology. \\ 

\subsection{Commuting Hamiltonians on rational SFT homology}
Following \cite{EGH},\cite{F3} we start with reviewing the appearance of the commuting Hamiltonian systems in the rational SFT. \\

Rational SFT is a generalization of contact (co)homology in the sense that its definition involves moduli spaces of holomorphic maps with not just arbitrary many negative ends, but also arbitrary many positive ends, see \cite{EGH} for details. For two ordered sets $\Gamma^+=(\gamma_0^+,\ldots,\gamma_{r^+}^+)$ and $\Gamma^-=(\gamma_0^-,\ldots,\gamma_{r^-}^-)$ of closed orbits $\gamma_1^{\pm},\ldots,\gamma_{r^{\pm}}^{\pm}$ the moduli space $\IM_r(\Gamma^+;\Gamma^-)$ consists of equivalence classes of tuples $(\tilde{u},(z_1^+,\ldots,z_{r^+}^+),(z_1^-,\ldots,z_{r^-}^-),(z_1,\ldots,z_r))$ with equivalence relation given by the action of the automorphisms of the domain and the $\IR$-action on the target, see \cite{EGH} for details. Here $(z_1^+,\ldots,z_{r^+}^+)$, $(z_1^-,\ldots,z_{r^-}^-)$ and $(z_1,\ldots,z_r)$ are disjoint collections of marked points on $S^2$ and $\tilde{u}=(h,u)$ is a $J$-holomorphic map from the punctured sphere $\Si=S^2\backslash\{z_1^{\pm},\ldots,z_{r^{\pm}}^{\pm}\}$ to $\IR\times M_{\phi}\cong S^1\times M$ converging to the closed orbits $\gamma_i^{\pm}$ in the positive/negative cylindrical end near $z_i^{\pm}$, $i=1,\ldots,r^{\pm}$. In particular, the induced map $h:\Si\to\IR\times S^1$ again defines a branched covering from $S^2$ to itself with branch points $z_i^{\pm}$ of order $k_i^{\pm}$, $i=1,\ldots,r^{\pm}$ over $\infty$ and $0$, respectively. Note that when $\Gamma^+$ consists of a single orbit $\gamma^+$ and there no additional marked points ($r=0$), then we just get back the moduli spaces of contact homology from before, $\IM^{\gamma^+}(\Gamma^-)=\IM_0(\gamma^+;\Gamma^-)$. \\

As in Gromov-Witten theory we can use the additional marked points $z_1,\ldots,z_r$ to define evaluation maps $$\ev=(\ev_1,\ldots,\ev_r): \IM_r(\Gamma^+;\Gamma^-) \to M_{\phi}^r\cong (S^1\times M)^r$$ given by $$\ev_i(\tilde{u},(z_1^{\pm},\ldots,z_{r^{\pm}}^{\pm}),(z_1,\ldots,z_r))\mapsto \tilde{u}(z_i)=(h(z_i),u(z_i)),\; i=1,\ldots,r,$$ which can be used to pullback differential forms from the target. Recalling that the cohomology ring of the Hamiltonian mapping torus $M_{\phi}\cong S^1\times M$ is given by $$H^*(M_{\phi}) \,=\, H^0(S^1)\otimes H^*(M) \,\oplus\,  H^1(S^1)\otimes H^{*-1}(M), $$ we now choose a string of differential forms $\theta_1,\ldots,\theta_K$ on $M_{\phi}$, where we will assume that the forms represent a basis of the first summand, $\theta_{\alpha}\in H^0(S^1)\otimes H^*(M) \cong H^*(M)$. \\

In contrast to the case of contact homology described before, we assign to each closed orbit $\gamma$ now \emph{two} formal graded variables $p_{\gamma}$ and $q_{\gamma}$ with $|p_{\gamma}|=|q_{\gamma}|\;\textrm{mod} 2$. As in Gromov-Witten theory we further assign to each cohomology class $\theta_{\alpha}$, $\alpha=1,\ldots,K$ a formal graded variable $\tau_{\alpha}$ with $|\tau_{\alpha}| = 2 - \deg \theta_{\alpha}\;\textrm{mod} 2$. They can again be viewed as coordinates of a super space $\V$, which contains the coordinate super space $\Q$ of contact homology after setting $p=0=t$ for $p=(p_{\gamma})$, $\tau=(\tau_{\alpha})$, but now carries a natural symplectic super-form $\sum_{\gamma} dp_{\gamma}\wedge dq_{\gamma}$  (in the formal sense). In particular, the space of functions $\PP=\T^{(0,0)}\V$ now carries a (graded) Poisson bracket $\{\cdot,\cdot\}:\PP\otimes\PP\to\PP$. \\

Using the moduli spaces defined above we can define the rational SFT Hamiltonian $\Ih\in\PP=\T^{(0,0)}\V$ of (rational) SFT as the sum over all $\Gamma^+$, $\Gamma^-$, $I$, where the coefficient in front of the monomial $\tau^Ip^{\Gamma^+}q^{\Gamma^-}t^{\omega(A)}$ with $p^{\Gamma^+}=p_{\gamma^+_1} \ldots p_{\gamma^+_{r^+}}$, $q^{\Gamma^-}=q_{\gamma^-_1} \ldots q_{\gamma^-_{r^-}}$, $\tau^I=\tau_{\alpha_1} \ldots \tau_{\alpha_r}$ is given by 
\begin{equation*}
 \frac{1}{r! r^+! r^-!}\frac{1}{\kappa^{\Gamma^+}\kappa^{\Gamma^-}} \int_{\IM_r(\Gamma^+;\Gamma^-;A)}
 \ev_1^*\theta_{\alpha_1}\wedge\ldots\wedge\ev_r^*\theta_{\alpha_r}.
\end{equation*}

It was shown in \cite{F3} that for each differential form one can define an infinite sequence of commuting Hamiltonians $\Ih_{\alpha,j}\in\T^{(0,0)}\V$. After introducing a special additional marked point $z_0$, we use that as well-known in Gromov-Witten theory there is a tautological line bundle $\LL$ over each moduli space $\IM_{r+1}(\Gamma^+;\Gamma^-)$, whose fibre over $(\tilde{u},z_0)= (h,u,(z_1^{\pm},\ldots,z_{r^{\pm}}^{\pm}),(z_0,z_1,\ldots,z_r))$ is given by the cotangent line $\LL_{(\tilde{u},z_0)} = T^*_{z_0}\Si$, and which extends smoothly over the compactified moduli space. Since the SFT moduli spaces have codimension-one boundary,  as discussed in subsection 2.3 it however does not make sense to integrate powers of the Euler class (= first Chern class) over the moduli space. Instead we have introduced in \cite{F3} the notion of (generic) coherent collections of (multi-valued) sections $(s)$ in the tautological line bundles over all moduli spaces, whose zero sets $\IM^1_{r+1}(\Gamma^+,\Gamma^-)=s^{-1}(0)$ in all $\IM_{r+1}(\Gamma^+;\Gamma^-)$ can be viewed as (Poincare dual) of a coherent Euler class involving all moduli spaces at once, see \cite{F3} for details.  \\

Choosing $j$ generic coherent collections of sections $(s_j)$ and defining $$\IM^j_{r+1}(\Gamma^+,\Gamma^-)=s_1^{-1}(0)\cap\ldots\cap s_j^{-1}(0) \subset \IM_{r+1}(\Gamma^+;\Gamma^-),$$ we define the desired sequence of Hamiltonians $\Ih_{\alpha,j}\in\T^{(0,0)}\V$ again as the sum over all $\Gamma^+$, $\Gamma^-$, $I$, where the coefficient in front of the monomial $ \tau^Ip^{\Gamma^+}q^{\Gamma^-}t^{\omega(A)}$ is now given by
\begin{equation*}
\frac{1}{r! r^+! r^-!} \frac{1}{\kappa^{\Gamma^+}\kappa^{\Gamma^-}}\int_{\IM^j_{r+1}(\Gamma^+,\Gamma^-,A)}
 \ev_0^*(\theta_{\alpha}\wedge dt)\wedge \ev_1^*\theta_{\alpha_1}\wedge\ldots\wedge\ev_r^*\theta_{\alpha_r},
\end{equation*}
with the canonical one-form $dt\in H^1(S^1) = H^1(S^1)\otimes H^0(M)\subset H^1(M_{\phi})$ given by the $S^1$-coordinate $t$ on $M_{\phi}\cong S^1\times M$. \\

It was shown in \cite{EGH} that the rational SFT Hamiltonian $\Ih\in\PP$ satisfies the master equation $\{\Ih,\Ih\}=0$. Similar as in contact homology this is equivalent to the fact that the symplectic gradient $X=X^{\Ih}\in\T^{(1,0)}\V$ of $\Ih$ with respect to the above formal symplectic super-form is an odd homological vector field on $\V$ and hence again defines a differential graded manifold $\V_X=(\V,X)$ with tensor fields $\T^{(r,s)}\V_X:=H^*(\T^{(r,s)}\V,\LL_X)$. Furthermore, as in contact homology, for two different choices of auxiliary data like almost complex structure and (domain-dependent) Hamiltonian perturbations, the resulting differential graded manifolds $(\V^+,X^+)$ and $(\V^-,X^-)$ are isomorphic.  The following generalisation of this result for the commuting Hamiltonians was shown in \cite{F3}. \\

\begin{theorem} Since $\{\Ih,\Ih_{\alpha,j}\}=X(\Ih_{\alpha,j})=0$, the Hamiltonians $\Ih_{\alpha,j}$ descend to  a sequence of functions on the differential graded manifold given by $(\V,X)$, which pairwise commute with respect to the Poisson bracket on $\T^{(0,0)}\V_X$, $$\{\Ih_{\alpha,j},\Ih_{\beta,k}\}=0\in\T^{(0,0)}\V_X.$$ Furthermore they are independent under choices of auxiliary data like almost complex structure and Hamiltonian perturbations in the sense that the under the isomorphism of rational SFT established in \cite{EGH} they get mapped to each other, $$\T^{(0,0)}\V^+_{X^+}\stackrel{\cong}{\rightarrow}\T^{(0,0)}\V^-_{X^-},\;\Ih^+_{\alpha,j}\mapsto\Ih^-_{\alpha,j},\;\alpha=1,\ldots,K, j\in\IN.$$
\end{theorem}
$ $\\

\subsection{Relation with Dubrovin's integrable hierarchies}
We now want to discuss the relation of the commuting Hamiltonian systems of the SFT of $M_{\phi}$ with the Floer theory of the underlying symplectomorphism $\phi$ as well as with the (dispersionless) integrable hierarchies that Dubrovin-Zhang assign to the Frobenius manifold of rational Gromov-Witten theory of the underlying symplectic manifold in \cite {DZ}. \\

Note that in the case when the symplectomorphism $\phi$ is the identity, it was already observed in \cite{EGH} that the SFT of $S^1\times M$ (more in general, of every circle bundle $S^1\to V\to (M,\omega))$ is determined by the Gromov-Witten theory of the underlying symplectic manifold $(M,\omega)$. This in turn was used in \cite{E}, see also \cite{R}, to prove that the commuting Hamiltonian system $\Ih_{\alpha,j}$ for $S^1\times M$ indeed agrees with the integrable system obtained from the Frobenius manifold determined by the Gromov-Witten potential of $(M,\omega)$. \\

In view of the fact that the Floer theory of a symplectomorphism $\phi$ generalizes the Gromov-Witten theory of $M$ in the same way as the SFT of its mapping torus $M_{\phi}$ generalizes the SFT of $S^1\times M$, we obtain following Floer generalization of this result. \\

With the above choice of differential forms $\theta_1,\ldots,\theta_K$ and making use of the canonical one-form $dt\in H^1(S^1\times M)$ as in the definition, every commuting Hamiltonian $\Ih_{\alpha,j}$ is counting Floer solution with additional marked points. For the following statement we restrict to the case $j=0$ as in \cite{EGH}, the generalization to arbitrary $j\in\IN$ is then obvious.

\begin{proposition}  The coefficient of the monomial $\tau^Ip^{\Gamma^+}q^{\Gamma^-}t^{\omega(A)}$ in $\Ih_{\alpha,0}$ is given by the integral of the pullback $\ev_0^*\theta_{\alpha}\wedge \ev_1^*\theta_{\alpha_1}\wedge\ldots\wedge\ev_r^*\theta_{\alpha_r}$ (no $dt$ !) of forms over the moduli space of Floer solutions $(u,(z_1^{\pm},\ldots,z_{r^{\pm}}^{\pm}),(z_0,z_1,\ldots,z_r))$ with $r$ additional marked points. More precisely, $u:\Si\to M$ satisfies the Floer equation $\CR_{J,H,h}(u)=\Lambda^{0,1}(du+X^H_{h_2}\otimes dh_2)$, where the holomorphic map $h=(h_1,h_2):\Si=S^2\backslash\{z_1^{\pm},\ldots,z_{r^{\pm}}^{\pm}\}\to\IR\times S^1$ is now determined by the requirement $h(z_0)=(0,0)\in\IR\times S^1$. \end{proposition}

Instead of selecting three marked points as in the definition of the big pair-of-pants product, we now use that integrating the pullback $\ev_0^*dt$ of the canonical one-form on the mapping torus is equivalent to requiring that for every element $u=(u,(z_1^{\pm},\ldots,z_{r^{\pm}}^{\pm}),(z_0,\ldots,z_r))$ the special marked point $z_0$ (used to define $\ev_0$) gets mapped to a fixed point on $S^1$ under the induced map $h=\pi\circ u:\Si=S^2\backslash\{z_1^{\pm},\ldots,z_{r^{\pm}}^{\pm}\}\to\IR\times S^1$. Note that when $\phi$ is the identity map, then $h_0^*(\IR\times M_{\phi})=\Si\times M$ and $u:\Si\to M$ is a holomorphic sphere in $M$ with additional marked points $z_1^{\pm},\ldots,z_{r^{\pm}}^{\pm},z_0,z_1,\ldots,z_r$. \\

On the other hand, it immediately follows from theorem 2.17 that for any chosen Hamiltonian symplectomorphism the system of commuting Hamiltonians is naturally isomorphic to the system of commuting Hamiltonians when the Hamiltonian is equal to zero. Since in the latter case it is shown in \cite{Ro} that the system of commuting Hamiltonians agrees with the dispersionless integrable hierarchy that Dubrovin-Zhang assign to the Frobenius manifold of rational Gromov-Witten theory of the underlying symplectic manifold in \cite {DZ}, we obtain the following immediate

\begin{corollary} The isomorphism $\V_X\cong\Lambda\QH^*(M)$ maps the system of commuting Hamiltonians $\Ih_{\alpha,j}\in \T^{(0,0)}\V_X$ to the integrable system defined by Dubrovin using the Frobenius manifold structure on $\QH^*(M)$ given by all rational Gromov-Witten invariants of the underlying symplectic manifold. \end{corollary}

We end this section with a short discussion about how the commuting Hamiltonians on rational SFT help us to find a substitute for the rational Gromov-Witten potential in the Floer theory of a Hamiltonian symplectomorphism. \\

For this let $F\in\T^{(0,0)}\V$ be the generating function, whose coefficient in front of the monomial $\tau^Ip^{\Gamma^+}q^{\Gamma^-}t^{\omega(A)}$ is given by the integral of the pullback $\ev_1^*\theta_{\alpha_1}\wedge\ldots\wedge\ev_r^*\theta_{\alpha_r}$ of forms over the moduli space of Floer solutions $ (u,(z_1^{\pm},\ldots,z_{r^{\pm}}^{\pm}),(z_0,z_1,\ldots,z_r))$ determined by $\Gamma^+$ and $\Gamma^-$. While we claim that $F$ does \emph{not} lie in the kernel of the symplectic vector field $X\in\T^{(1,0)}\V$ of rational SFT, we claim that the above generating function agrees with the first descendant Hamiltonian $\Ih_{0,1}$ for the canonical zero-form up to a natural weighting factor in front of its summands. \\

Indeed, since by the above theorem $\Ih_{0,1}$ counts holomorphic sections with one additional marked point carrying one psi class, it follows from (an analogue of) the dilaton equation that the coefficient of $\Ih_{0,1}$ in front of each monomial agrees with the coefficient of $F$ multiplied with the Euler characteristic of the underlying punctured sphere. While $\Ih_{0,1}$ and $F$ hence carry the same geometrical information (since the Euler characteristic is nonzero for spheres with three or more marked points), the first descendant Hamiltonian $\Ih_{0,1}$ (in contrast to $F$) indeed defines an invariant of the symplectomorphism $\phi$.

\end{document}